\documentclass{amsart}

\usepackage{fancyhdr}

\usepackage{amsmath,amssymb,amsthm,mathrsfs}%,refcheck}

\usepackage[colorlinks=true,linkcolor=blue]{hyperref}
\makeatletter
\renewcommand*{\eqref}[1]{%
  \hyperref[{#1}]{\textup{\tagform@{\ref*{#1}}}}%
}
\makeatother

\usepackage{enumitem,natbib,nicefrac}

\setcitestyle{numbers,open={[},close={]}}

\numberwithin{equation}{section}

\def\<{{\langle}}
\def\>{{\rangle}}
\def\R{{\mathbb R}}

\def\d{{\rm d}}

\def\:{\colon }

\def\dist{{\rm dist}}
\def\A{{\mathcal A}}
\def\P{{\mathcal P}}
\def\be#1{\begin{equation}\label{#1}}
\def\ee{\end{equation}}

\def\sbsection#1{\subsection{#1}$ $\smallskip}

\newtheorem{theorem}{Theorem}[section]

\newtheorem{corollary}[theorem]{Corollary}
\newtheorem{lemma}[theorem]{Lemma}

\providecommand{\abs}[1]{\left\lvert{#1}\right\rvert}

\begin{document}

\title{Explicit characterisation of the fractional power spaces of the Dirichlet Laplacian and Stokes operators}

\author{Karol W. Hajduk}
\address{Department of Mathematics and Statistics, Faculty of Science, Masaryk University,
Building 08, Kotl\'a{\v r}sk\'a 2, 611 37, Brno, Czech Republic}
%\curraddr{Mathematics Institute, Zeeman Building, University of Warwick, Coventry, CV4 7AL, United Kingdom}
\email{\href{mailto:hajduk@math.muni.cz}{hajduk@math.muni.cz}}
\thanks{KWH was supported by an EPSRC Standard DTG EP/M506679/1 and by the Warwick Mathematics Institute.}
\author{James C. Robinson}
\address{Mathematics Institute, Zeeman Building, University of Warwick,
Coventry, CV4 7AL, United Kingdom}
\email{\href{mailto:j.c.robinson@warwick.ac.uk}{j.c.robinson@warwick.ac.uk}}

%    General info

\date{August 06, 2021.}


\begin{abstract}
  We identify explicitly the fractional power spaces for the $L^2$ Dirichlet Laplacian and Dirichlet Stokes operators using the theory of real interpolation. The results are not new, but we hope that our arguments are relatively accessible.
\end{abstract}

%%    Information for second author
%\author{Karol W. Hajduk\thanks{KWH was supported by an EPSRC Standard DTG EP/M506679/1 and by the University of Warwick.}\\
%\small Department of Mathematics and Statistics, Faculty of Science,\\ Masaryk University,
%\small Building 08, Kotl\'a{\v r}sk\'a 2, 611 37, Brno. Czech Republic.\\
%\small{\href{mailto:hajduk@math.muni.cz}{\tt hajduk@math.muni.cz}}\\$ $\\
%%\small{\tt hajduk@math.muni.cz}\\$ $\\
%James C. Robinson\thanks{Corresponding author.}\\
%\small Mathematics Institute, Zeeman Building,\\
%\small University of Warwick, Coventry, CV4 7AL. United Kingdom.\\
%\small{\href{mailto:j.c.robinson@warwick.ac.uk}{\tt j.c.robinson@warwick.ac.uk}}}
%
%%\thanks{JCR was supported in part by an EPSRC Leadership Fellowship EP/G007470/1.}

\maketitle

\section{Introduction}

%\makeatletter
%    \renewcommand{\theequation}{{\thesection}.\@arabic\c@equation}
%    \renewcommand{\thetheorem}{{\thesection}.\@arabic\c@theorem}
%    \renewcommand{\thedefinition}{{\thesection}.\@arabic\c@definition}
%    \renewcommand{\thecorollary}{{\thesection}.\@arabic\c@corollary}
%    \renewcommand{\theproposition}{{\thesection}.\@arabic\c@proposition}
%\makeatother

In many applications it is useful to have an explicit characterisation of the fractional power spaces of the Dirichlet Laplacian. (In our case this was prompted by a new approximation scheme that allows for simultaneous approximation in $L^2$-based Sobolev spaces and $L^p$ spaces, using weighted truncations of the eigenfunction expansion, see Fefferman, Hajduk, \& Robinson \cite{FHR}.) 

The results we present here are not new, but straightforward proofs are hard to find in the literature. Additionally, we use in our arguments the theory of real interpolation spaces rather than the complex interpolation used by other authors. The characterisation of the domains of the Dirichlet Laplacian can be found in the papers by Grisvard \cite{Grisvard}, Fujiwara \cite{Fujiwara67}, and Seeley \cite{Seeley3}. Note that  Fujiwara's statement is not correct for $\theta=3/4$, and that  Seeley also gives the corresponding characterisation for the operators in $L^p$-based spaces. For the Stokes operator $\A$, Giga \cite{Giga85} and Fujita \& Morimoto \cite{Fujita-Morimoto} both show that $D(\A)=D(A)\cap H_\sigma$; the former in the greater generality of $L^p$-based spaces. We use a key idea from the proof of Fujita \& Morimoto in our argument in Section \ref{sec:fracStokes}.

\section{Fractional power spaces of linear operators}\label{sb-XpDr}

%\makeatletter
%    \renewcommand{\theequation}{{\thesection}.\@arabic\c@equation}
%    \renewcommand{\thetheorem}{{\thesection}.\@arabic\c@theorem}
%    \renewcommand{\thedefinition}{{\thesection}.\@arabic\c@definition}
%    \renewcommand{\thecorollary}{{\thesection}.\@arabic\c@corollary}
%    \renewcommand{\theproposition}{{\thesection}.\@arabic\c@proposition}
%\makeatother

We suppose that $H$ is a separable Hilbert space, with inner product $\<\cdot,\cdot\>$ and norm $\|\cdot\|$, and that $A$ is a positive, self-adjoint operator on $H$ with compact inverse. In this case $A$ has a complete set of orthonormal eigenfunctions $\{w_n\}$ with corresponding eigenvalues $\lambda_n>0$, which we order so that $\lambda_{n+1}\ge\lambda_n$.

Recall that for any $\alpha\ge0$ we can define $D(A^\alpha)$ as the subspace of $H$ where
\begin{equation}\label{As-def}
D(A^\alpha) :=\left\{u=\sum_{j=1}^\infty \hat u_jw_j:\ \sum_{j=1}^\infty\lambda_j^{2\alpha}|\hat u_j|^2<\infty\right\}.
\end{equation}
For $\alpha<0$ we can take this space to be the dual of $D(A^{-\alpha})$; the expression in \eqref{As-def} can then be understood as an element in the completion of the space of finite sums with respect to the $D(A^\alpha)$ norm defined below in \eqref{DAnorm}. For all $\alpha\in\R$ the space $D(A^\alpha)$ is a Hilbert space with inner product
$$
\<u,v\>_{D(A^\alpha)}:=\sum_{j=1}^\infty\lambda_j^{2\alpha}\hat u_j\hat v_j
$$
and corresponding norm
\be{DAnorm}
\|u\|_{D(A^\alpha)}^2 :=\sum_{j=1}^\infty\lambda_j^{2\alpha}|\hat u_j|^2
\ee
[note that $D(A^0)$ coincides with $H$]. We can define $A^\alpha\:D(A^\alpha)\to H$ as the mapping
$$
\sum_{j=1}^\infty \hat u_jw_j\mapsto \sum_{j=1}^\infty \lambda_j^\alpha\hat u_jw_j,
$$
and then $\|u\|_{D(A^\alpha)}=\|A^\alpha u\|$. Note that $A^\alpha$ also makes sense as a mapping from $D(A^\beta)\to D(A^{\beta-\alpha})$ for any $\beta\in\R$, and that for $\beta\ge\alpha\ge0$ we have
\begin{equation}\label{higherpower}
D(A^\beta)=\{u\in D(A^{\beta-\alpha}):\ A^{\beta-\alpha}u\in D(A^\alpha)\}.
\end{equation}

\section{Characterisation results}

We will prove the following theorem, which gives the explicit form of these fractional power spaces for the Dirichlet Laplacian and the Dirichlet Stokes operator. It combines the results of Lemma \ref{half}, Corollaries \ref{0thalf} and \ref{halft1}, and Lemma \ref{AnA}.

\begin{theorem}\label{fp-Dirichlet}
Let $\Omega$ be an open, bounded domain with smooth boundary. When $A$ is the negative Dirichlet Laplacian on $\Omega\subset\R^d$, $d\ge2$, we have
  $$
  D(A^\theta)=\begin{cases}
  H^{2\theta}(\Omega), &0<\theta< 1/4, \\
  H^{1/2}_{00}(\Omega),&\theta=1/4,\\
  H^{2\theta}_0(\Omega), &1/4<\theta\le 1/2, \\
  H^{2\theta}(\Omega)\cap H_0^1(\Omega), &1/2<\theta\le 1,
  \end{cases}
  $$
  where $H_{00}^{1/2}(\Omega)$ consists of all $u\in H^{1/2}(\Omega)$ such that
    $$
    \int_\Omega\rho(x)^{-1}|u(x)|^2\,\d x<\infty,
    $$
    with $\rho(x)$ any $C^\infty$ function comparable to $\dist(x,\partial\Omega)$. If $A$ is the Stokes operator on $\Omega$ with Dirichlet boundary conditions then the domains of the fractional powers of $A$ are as above, except that all spaces are intersected with
$$
H_\sigma:=\mbox{completion of }\{\phi\in [C_c^\infty(\Omega)]^d:\ \nabla\cdot\phi=0\}\mbox{ in the norm of }L^2(\Omega).
$$
\end{theorem}

\section{Fractional power spaces and real interpolation}

We begin with a very quick treatment of the fractional powers of positive unbounded self-adjoint compact-inverse operators on a Hilbert space; in this case it is easy to show that the fractional power spaces are given as real interpolation spaces (cf.\ Chapter 1 of Lions \& Magenes \cite{LionsMagenes1}, from which we quote a number of results in what follows).

\sbsection{Real interpolation (`$K$-method')}

We recall the method of `real interpolation', due to Lions \& Peetre \cite{Lions59,LionsPeetre} as adopted by Lions \& Magenes; their $\theta$-intermediate space corresponds to the $(\theta,2; K)$ interpolation space in the more general theory covered in  Adams \& Fournier \cite{AdamsFournier} or Lunardi \cite{Lunardi}, for example.

We suppose that $X$ and $Y$ are Banach spaces, both continuously embedded in some Hausdorff topological vector space $B$. For any $u\in X+Y$ we define
\begin{equation}\label{ourK}
K(u,t) := \inf_{x+y=u}\left(\|x\|^2_X+t^2\|y\|_{Y}^2\right)^{1/2};
\end{equation}
we follow \cite{LionsMagenes1} in choosing this particular form for $K$. We define
\be{is-def}
(X,Y)_\theta:=\left\{u\in X+Y\: t^{-\theta}K(u,t)\in L^2(0,\infty;\textstyle{\nicefrac{\d t}{t}})\right\};
\ee
this is a Banach space with norm
$$
\|u\|_\theta:=\left\|t^{-\theta}K(u,t)\right\|_{L^2(0,\infty;\nicefrac{\d t}{t})}.
$$
[Since $a^2+t^2b^2\le(a+tb)^2\le 2(a^2+t^2b^2)$ this is equivalent to the standard definition of the space $(X,Y)_{\theta,2; K}$, in which $K$ is not defined as in \eqref{ourK} but rather as $K(u,t):=\inf_{x+y=u}\|x\|_X+t\|y\|_Y$. The definition we adopt here is more suited to the Hilbert space case.]

\sbsection{Fractional power spaces via real interpolation}

 We now give a simple proof that the fractional power spaces of $A$ are given by real interpolation spaces when $A$ is a positive unbounded self-adjoint operator with compact inverse (cf.\ Theorem I.15.1 in \cite{LionsMagenes1}).

\begin{lemma}\label{fps}
  Suppose that $A$ is a positive unbounded self-adjoint operator with compact inverse and domain $D(A)$ in a Hilbert space $H$ (as in Section \ref{sb-XpDr}). Then
  \be{fpri}
  (H,D(A))_{\theta} = D(A^\theta), \qquad 0<\theta<1.
  \ee
\end{lemma}

\noindent  (A similar result holds for general positive self-adjoint operators on Hilbert spaces. One can obtain \eqref{fpri} using complex interpolation provided that the imaginary powers of $A$ are bounded, which they are in this case (see Seeley \cite{Seeley2}); since real and complex interpolation spaces coincide for Hilbert spaces (see Chapter 1 of Triebel \cite{Triebel}), \eqref{fpri} then holds using real interpolation in this more general setting; for a related discussion see Chapter I, Section 2.9 in the book by Amann \cite{Amann}. See also the two papers \cite{Seeley1,Seeley3} by Seeley.)

\begin{proof}

For $u=\sum_{j=1}^\infty \hat u_jw_j$ we have
$$
K(u,t)=\inf_{(y_j)}\left[\sum_{j = 1}^{\infty} \abs{\hat u_j-y_j}^2+t^2\lambda_j^2\abs{y_j}^2\right]^{1/2}.
$$
A simple minimisation over $(y_j)$ shows that
$$
K(u,t)=\left(\sum_{j = 1}^{\infty}\frac{t^2\lambda_j^2 \abs{\hat u_j}^2}{1+t^2\lambda_j^2}\right)^{1/2}.
$$
Now observe that
\begin{align*}
\int_{0}^{\infty} t^{-2\theta}K(u,t)^2\,\frac{\d t}{t}&=\int_{0}^{\infty} \sum_{j = 1}^{\infty}\frac{(t^2\lambda_j^2)^{1-\theta}}{1+t^2\lambda_j^2}\,\lambda_j^{2\theta}\abs{\hat u_j}^2\,\frac{\d t}{t}\\
&=\sum_{j = 1}^{\infty}\int_{0}^{\infty}\frac{(t^2\lambda_j^2)^{1-\theta}}{1+t^2\lambda_j^2}\,\lambda_j^{2\theta}\abs{\hat u_j}^2\,\frac{\d t}{t}\\
&=\sum_{j = 1}^{\infty}\lambda_j^{2\theta}\abs{\hat u_j}^2\int_{0}^{\infty}\frac{s^{1-2\theta}}{1+s^2}\,\d s\\
&=I(\theta)\sum_{j = 1}^{\infty}\lambda_j^{2\theta}\abs{\hat u_j}^2 = I(\theta)\|A^{\theta}u\|^2\\
&=I(\theta)\|u\|_{D(A^\theta)}^2,
\end{align*}
where
$$
I(\theta) = \int_0^\infty\frac{s^{1-2\theta}}{1+s^2}\,\d s<\infty
$$
for $0 < \theta < 1$. (In fact the integral can be evaluated explicitly using contour integration to give $I(\theta)=\frac{\pi}{2}\frac{1}{\sin(\pi\theta)}$.) It follows that
$u\in (H,D(A))_\theta$ if and only if $u\in D(A^\theta)$.
  \end{proof}

The following particular cases of the `reiteration theorem' \cite[Theorem 1.6.1]{LionsMagenes1} are simple corollaries of the above result.

  \begin{corollary}\label{reit}
  In the same setting as that of Lemma \ref{fps}
  $$
  (H,D(A^{1/2}))_{\theta} = D(A^{\theta/2})\qquad\mbox{and}\qquad (D(A^{1/2}),D(A))_{\theta} = D(A^{(1+\theta)/2}).
  $$
  \end{corollary}

  \begin{proof}
    For the first equality we apply Lemma \ref{fps} with $A$ replaced by $A^{1/2}$; for the second we apply Lemma \ref{fps} with $A$ replaced by $A^{1/2}$ and the `base space' $H$ replaced by $D(A^{1/2})$, and note that
    $$D(A^{(1+\theta)/2})=\{u\in H:\ A^{1/2}u\in D(A^{\theta/2})\}. \eqno \qedhere$$
  \end{proof}

  To obtain fractional powers of operators with boundary conditions, or other constraints (e.g.\ the divergence-free constraint associated with the Stokes operator) the following simple result will be useful: it provides one way to circumvent the fact that interpolation does not respect intersections, i.e.\ in general
  $$
  (X\cap Z,Y\cap Z)_\theta\neq(X,Y)_\theta\cap Z.
  $$
  A similar result can be found as Theorem 14.3 in \cite{LionsMagenes1} in the context of holomorphic interpolation.

  (A related result can be found as Proposition A.2 in Rodr{\'\i}guez-Bernal \cite{ARB}.)

\begin{lemma}\label{isitnew}
     Let $(H,\|\cdot\|_H)$ and $(D,\|\cdot\|_D)$ be Hilbert spaces, with $H_0$ a Hilbert subspace of $H$ (i.e.\ with the same norm) and $D\subset H$ with continuous inclusion. Suppose that there exists a bounded linear map $T\:H\to H_0$ such that $T|_{H_0}$ is the identity and $T|_{D} \: D \to D \cap H_0$ is also bounded, in the sense that
     $$
     \|Tf\|_D\le C\|f\|_D\qquad\mbox{for some }C>0\mbox{ for every }f\in D.
     $$
     Then for every $0<\theta<1$
 $$
 (H_0,D\cap H_0)_\theta=(H,D)_\theta\cap H_0
 $$
 with  norm equivalent to that in $(H,D)_\theta$.
\end{lemma}

 \begin{proof}

  Since $H_0\subset H$ and $D\cap H_0\subset D$, it follows from the definition \eqref{is-def} of the interpolation spaces that
    $$
    (H_0, D \cap H_0)_\theta\subset (H,D)_\theta\cap H_0
    $$
    with
    $$
    \|u\|_{\theta}\le C\|u\|_{0,\theta},
    $$
    where $\|\cdot\|_{0,\theta}$ denotes the norm in $(H_0,D\cap H_0)_\theta$ (and $\|\cdot\|_\theta$ is the norm in $(H,D)_\theta$).

    Now suppose that $u\in (H,D)_\theta\cap H_0$; then for each $t>0$ we can find $f(t)\in H$ and $g(t)\in D$ such that we can write
    $$
    u=f(t)+g(t)\qquad\mbox{with}\qquad \|f(t)\|_H^2+t^2\|g(t)\|_D^2\le 2K(u,t)^2;
    $$
    then
    $$
    \int_{0}^{\infty} t^{-2\theta-1}\left(\|f(t)\|_H^2+t^2\|g(t)\|_D^2\right)\,\d t\le 2\|u\|_\theta^2.
    $$
    Now since $u\in H_0$ and $T|_{H_0}={\rm Id}$ we also have
    $$
    u=Tu=Tf(t)+Tg(t),
    $$
    with $Tf(t)\in H_0$ and $Tg(t)\in D\cap H_0$, so that
    \begin{align*}
    \|u\|_{0,\theta}^2&\le\int_{0}^{\infty} t^{-2\theta-1}\left(\|Tf(t)\|_H^2+t^2\|Tg(t)\|_D^2\right)\,\d t\\
    &\le C^2\int_{0}^{\infty} t^{-2\theta-1}\left(\|f(t)\|_H^2+t^2\|g(t)\|_D^2\right)\,\d t\\
    &\le 2C^2\|u\|_\theta^2,
    \end{align*}
    i.e. $\|u\|_{0,\theta}\le C'\|u\|_\theta$, from which the conclusion follows.
  \end{proof}

\section{Identifying fractional power spaces}

\sbsection{Sobolev spaces and interpolation}

  We first recall how fractional Sobolev spaces are defined using interpolation, and some of their properties. It is then relatively straightforward to give explicit characterisations of the fractional power spaces of the Dirichlet Laplacian and the Stokes operator.

For non-integer $s$ the space $H^s(\Omega)$ is defined by setting
$$
H^{k\theta}(\Omega):=(L^2(\Omega),H^k(\Omega))_\theta,\qquad 0<\theta<1,
$$
for any integer $k$ (equation (I.9.1) in \cite{LionsMagenes1}); this definition is independent of $k$ and is consistent with the standard definition whenever $k\theta$ is an integer, so we have
\begin{equation}
(H^{s_1}(\Omega),H^{s_2}(\Omega))_\theta=H^{(1-\theta)s_1+\theta s_2}(\Omega),\qquad s_1<s_2,\ 0<\theta<1,\label{Hsinterp}
\end{equation}
see \cite[Theorem I.9.6]{LionsMagenes1}. Defined in this way $H^s(\Omega)$ is the set of restrictions to $\Omega$ of functions in $H^s(\R^n)$ \cite[Theorem I.9.1]{LionsMagenes1}.

For all $s\ge0$ we define
$$
H_0^s(\Omega):=\mbox{completion of }C_c^\infty(\Omega)\mbox{ in }H^s(\Omega);
$$
for $0\le s\le 1/2$ we have $H_0^s(\Omega)=H^s(\Omega)$ \cite[Theorem I.11.1]{LionsMagenes1}.

\sbsection{Fractional power spaces of Dirichlet Laplacian}\label{sec:fracLap}

We now consider the case when $A=-\Delta$ is the negative Dirichlet Laplacian on a bounded domain $\Omega$; to avoid technicalities we assume that $\partial\Omega$ is smooth.  From standard regularity results for weak solutions, see Theorem 8.12 in Gilbarg \& Trudinger \cite{GilbargTrudinger} or Section 6.3 in Evans \cite{Evans}, for example, we know that $D(A)=H^2(\Omega)\cap H_0^1(\Omega)$. The following result is well known, but we provide a proof (after the discussion following Proposition 4.5 in Constantin \& Foias \cite{ConstantinFoias}) for the sake of completeness.

\begin{lemma}\label{half}
If $A$ is the negative Dirichlet Laplacian on $\Omega$ then
$$D(A^{1/2})=H_0^1(\Omega).$$
\end{lemma}

\begin{proof} We have
\be{useful}
\<Au,v\>=\<\kern-3pt-\kern-3pt\Delta u,v\>=\<\nabla u,\nabla v\>
\ee
whenever $u\in D(A)$ and $v\in H_0^1(\Omega)$, see the proof of Proposition 4.2 in Constantin \& Foias \cite{ConstantinFoias} (their proof is given for the Stokes operator, but it works equally well in the case of the Laplacian).

If we let $(w_j)$ and $(\lambda_j)$ be the eigenfunctions and corresponding eigenvalues of $A$, then $(w_j)$ form a basis for $L^2(\Omega)$ (so also for $H_0^1(\Omega)$) and since
$\lambda_j^{-1/2}w_j\in D(A)\subset H_0^1$ we can use \eqref{useful} to write
\begin{align*}
\delta_{jk}=\<\lambda_j^{-1/2}w_j,\lambda_k^{-1/2}w_k\>_{D(A^{1/2})}&=\<A(\lambda_j^{-1/2}w_j),\lambda_k^{-1/2}w_k\>\\
&=\<\nabla (\lambda_j^{-1/2}w_j),\nabla (\lambda_k^{-1/2}w_k)\>.
\end{align*}
It follows that $D(A^{1/2})$ is a closed subspace of $H_0^1$. [Recall from \eqref{As-def} that $D(A^{1/2})$ is defined as the collection of certain convergent eigenfunction expansions; the above equality shows that if this expansion converges in the $D(A^{1/2})$ norm then it also converges in the norm of $H_0^1$.]

If $v\in H_0^1$ with $\<v,u\>_{H_0^1}=0$ for all $u\in D(A^{1/2})$ then for every $j$
$$0=\<\nabla v,\nabla w_j\>=\<v,Aw_j\>=\lambda_j\<v,w_j\>$$
and so $v=0$, which shows that $D(A^{1/2})=H_0^1$.
\end{proof}

We can now appeal to results from Lions \& Magenes to deal with the range $0<\theta<1/2$.

\begin{corollary}\label{0thalf}
    If $A$ is the negative Dirichlet Laplacian on $\Omega$ then
$$
    D(A^\theta)=\begin{cases}H^{2\theta}(\Omega)&0<\theta< 1/4,\\
   H^{1/2}_{00}(\Omega)&\theta=1/4,\\
    H_0^{2\theta}(\Omega)&1/4<\theta<1/2.
    \end{cases}
    $$
   where $H_{00}^{1/2}(\Omega)$ consists of all $u\in H^{1/2}(\Omega)$ such that
    $$
    \int_\Omega\rho(x)^{-1}|u(x)|^2\,\d x<\infty,
    $$
    with $\rho(x)$ any $C^\infty$ function comparable to $\dist(x,\partial\Omega)$.
\end{corollary}

\begin{proof}
  We note that
  $$
  D(A^{\theta/2})=(H,D(A^{1/2}))_{\theta}=(L^2,H_0^1)_\theta,
  $$
  and then the expressions on the right-hand side follow immediately from \cite[Theorems I.11.6 and I.11.7]{LionsMagenes1}.
\end{proof}

  Note that the result above is relatively elementary for $\theta\neq 1/4$: since $w_j \in D(A^r)$ is a countable sequence whose linear span is dense in $D(A^s)$, $D(A^r)$ is always dense in $D(A^s)$ for $0\le s<r\le 1$; since Corollary \ref{reit}  shows that $D(A^{1/2})=H_0^1(\Omega)$, it follows that $H_0^1(\Omega)$ is dense in $D(A^\theta)$ for $\theta<1/2$, and so, since $\|u\|_{H^{2\theta}}\le C_\theta\|u\|_{D(A^\theta)}$,
  \begin{align*}
  D(A^\theta)&=\{\mbox{completion of }H_0^1(\Omega)\mbox{ in the norm of }D(A^\theta)\}\\
  &\subseteq\{\mbox{completion of }H_0^1(\Omega)\mbox{ in the norm of }H^{2\theta}(\Omega)\}\\
   &=\{\mbox{completion of }C_c^\infty(\Omega)\mbox{ in the norm of }H^{2\theta}(\Omega)\}=H^{2\theta}_0(\Omega).
  \end{align*}

    To show the equivalence of the $H^{2\theta}$ and $D(A^\theta)$ norms (and hence equality of $D(A^{\theta})$ and $H_0^{2\theta}$) note that functions in $L^2$, $H^s$ for $0<s<1/2$, and $H_0^s$ for $1/2<s\le 1$ can be extended by zero to functions in $H^s(\R^n)$ without increasing their norms \cite[Theorem I.11.4]{LionsMagenes1}; an argument following that of Example 1.1.8 in Lunardi \cite{Lunardi} then shows that the norms in $D(A^\theta)$ and in $H^{2\theta}$ are equivalent provided that $\theta\neq1/4$.

   Since functions in $H^{1/2}(\Omega)=H^{1/2}_0(\Omega)$ \textit{cannot} be extended by zero to functions in $H^{1/2}(\R^n)$ \cite[Theorem I.11.4]{LionsMagenes1} the case of $\theta=1/4$ is significantly more involved.

To deal with the range $1/2<\theta<1$ we will use the intersection lemma (Lemma \ref{isitnew}) and the following simple result.

\begin{lemma}\label{lm}
    Let $u\in H^s(\Omega)$ with $s=1$ or $s=2$, and let $w\in H_0^1(\Omega)$ solve
    \begin{equation}\label{weakone}
    \<\nabla w,\nabla\phi\>=\<\nabla u,\nabla\phi\>\qquad\mbox{for all}\quad\phi\in H_0^1(\Omega).
    \end{equation}
        Then $u\mapsto w$ is a bounded linear map from $H^s(\Omega)$ into $H^s(\Omega)\cap H_0^1(\Omega)$ and $w=u$ whenever $u\in H_0^1(\Omega)$.
\end{lemma}

  \begin{proof}
    The Riesz Representation Theorem guarantees that (\ref{weakone}) has a unique solution $w\in H_0^1(\Omega)$ for every $u\in H^1(\Omega)$. That $w=u$ when $u\in H_0^1(\Omega)$ is then immediate, and the choice $\phi=w$ guarantees that $\|\nabla w\|_{L^2}\le\|\nabla u\|_{L^2}$. To deal with the $s=2$ case, simply note that (\ref{weakone}) is the weak form of the equation
    $$
    -\Delta w=-\Delta u,\qquad w|_{\partial\Omega}=0,
    $$
    and standard regularity results for this elliptic problem (e.g.\ Section 6.3 in Evans \cite{Evans}) guarantee that $\|w\|_{H^2}\le C\|\Delta u\|_{L^2}\le C\|u\|_{H^2}$.
\end{proof}

We can now characterise $D(A^\theta)$ for $1/2<\theta<1$.

\begin{corollary}\label{halft1}
    If $A$ is the negative Dirichlet Laplacian on $\Omega$ then
  $$
  D(A^\theta)=  H^{2\theta}(\Omega)\cap H_0^1(\Omega)\qquad\mbox{for}\qquad 1/2<\theta<1.
$$
\end{corollary}

\begin{proof}
Corollary \ref{reit} guarantees that
$$
D(A^\theta)=(D(A^{1/2}),D(A))_{2\theta-1}=(H_0^1,H^2\cap H_0^1)_{2\theta-1}.
$$
Choosing $H=H^1(\Omega)$, $H_0=H_0^1(\Omega)$, and $D=H^2(\Omega)$ in Lemma \ref{isitnew}, we can let $T$ be the map $u\mapsto w$ defined in Lemma \ref{lm} to deduce that
$$
(H_0^1,H^2\cap H_0^1)_{2\theta-1}=(H^1,H^2)_{2\theta-1}\cap H_0^1=H^{2\theta}\cap H_0^1,
$$
using (\ref{Hsinterp}).
\end{proof}

\sbsection{Fractional power spaces of the Stokes operator}\label{sec:fracStokes}

Let $\P$ denote the `Leray projection', i.e.\ the orthogonal projection in $L^2(\Omega)$ onto
  $$
H_\sigma(\Omega):=\mbox{completion of }\{\phi\in C_c^\infty(\Omega):\ \nabla\cdot\phi=0\}\mbox{ in the norm of }L^2(\Omega).
$$
Since $\P$ is an orthogonal projection we have the symmetry property
\be{Psymm}
\<\P u,v\>=\<u,\P v\>\qquad\mbox{for every}\ u,v\in L^2(\Omega).
\ee

The Stokes operator $\A$ on $\Omega$ is defined as $\A :=\P A$, where $A$ is the negative Dirichlet Laplacian, and has domain
$$
D(\A)=H^2(\Omega)\cap H_0^1(\Omega)\cap H_\sigma(\Omega)=D(A)\cap H_\sigma(\Omega),
$$
see Theorem 3.11 in Constantin \& Foias \cite{ConstantinFoias}. It is a positive unbounded self-adjoint operator with compact inverse (see Chapter 4 in \cite{ConstantinFoias}), so still falls within the general framework we have considered above.

Now we show that $D(\A^\theta)=D(A^\theta)\cap H_\sigma$. We can do this using the `intersection lemma' (Lemma \ref{isitnew}) via an appropriate choice of the mapping $T$: our choice is inspired by the proof of this equality due to Fujita \& Morimoto \cite{Fujita-Morimoto}, who use the trace-based formulation of interpolation spaces.

\begin{lemma}\label{AnA}
  For every $0<\theta<1$ we have $D(\A^\theta)=D(A^\theta)\cap H_\sigma$ with $\|\A^\theta u\|$ and $\|A^\theta u\|$ equivalent norms on $D(\A^\theta)$; in particular, the inclusion $D(\A^\theta)\subset D(A^\theta)$ is continuous.
  \end{lemma}

\begin{proof}
First observe that Lemma \ref{fps} gives
  $$
  D(\A^\theta)=(H_\sigma,D(A)\cap H_\sigma)_\theta.
  $$

In order to apply the intersection result of Lemma \ref{isitnew} we consider the operator $\tilde T\:D(A)\to D(\A)$ defined by setting
$$
\tilde T:=\A^{-1}\P A.
$$
As an operator from $D(A)$ into $D(\A)$ this is bounded, due to elliptic regularity results for the Stokes operator ($\|\A^{-1}g\|_{H^2}\le C\|g\|_{L^2}$ for $g\in H_\sigma$, see Theorem 3.11 in Constantin \& Foias \cite{ConstantinFoias}, for example): for any $f\in D(A)$ we have
$$
\|\tilde Tf\|_{D(\A)}\le \|\tilde Tf\|_{H^2(\Omega)}\le C\|\P Af\|_{L^2(\Omega)}\le C\|Af\|_{L^2(\Omega)}= C\|f\|_{D(A)}.
$$

We now extend $\tilde T$ to an operator $T\:L^2\to H_\sigma$: if we take
$\psi\in H_\sigma$ and $\phi\in D(A)$ then, since both $\A$ and $A$ are self-adjoint and $\P$ is symmetric \eqref{Psymm},
  \begin{align*}
  |\<\psi,\tilde T\phi\>|&=  |\<\psi,\A^{-1}\P A\phi\>|=|\<\A^{-1}\psi,\P A\phi\>|\\
  &=|\<\A^{-1}\psi,A\phi\>|=|\<A \A^{-1}\psi,\phi\>|\\
  &\le \|\A^{-1}\psi\|_{H^2}\|\phi\|_{L^2}\\
  &\le C\|\psi\|_{L^2}\|\phi\|_{L^2},
  \end{align*}
  which shows that
  $$
  \|\tilde T\phi\|_{H_\sigma}\le C\|\phi\|_{L^2}.
  $$
  Since $\tilde T$ is linear and $D(A)$ is dense in $L^2$ it follows that we can extend $\tilde T$ uniquely to an operator $T\:L^2(\Omega)\to H_\sigma$ as claimed.

Note that $T$ is the identity on $H_\sigma$: this can be seen by expanding $u\in H_\sigma$ in terms of the eigenfunctions of $\A$.

We now obtain the result by applying Lemma \ref{isitnew} choosing $H=L^2(\Omega)$, $H_0=H_\sigma$, $D=D(A)$, and letting $T\:L^2\to H_\sigma$ be the operator we have just constructed.
\end{proof}

\bibliography{biblio}
\bibliographystyle{acm}
\end{document}